\documentclass[leqno, 12pt]{article}
\usepackage{amsmath}
\usepackage{amsthm}
\usepackage{amsfonts}
\usepackage{amssymb}
\usepackage{hyperref}

\makeatletter
 
  \@addtoreset{equation}{section}
 \makeatother

\newtheorem{thm}{Theorem}[section]
\newtheorem{lem}[thm]{Lemma}

\theoremstyle{definition}

\theoremstyle{remark}

\begin{document}
\title{Infinitary superperfect numbers\footnote{2010 Mathematics 
Subject Classification: 11A05, 11A25.}
\footnote{Key words: Odd perfect numbers; Super perfect numbers; Unitary divisors; Infinitary divisors; The sum of divisors.}}
\author{Tomohiro Yamada}
\date{}
\maketitle

\begin{abstract}
We shall show that $9$ is the only odd infinitary superperfect numbers.
\end{abstract}

\section{Introduction}\label{intro}
As usual, $\sigma(N)$ denotes the sum of divisors of a positive integer $N$.
$N$ is called to be perfect if $\sigma(N)=2N$.
It is a well-known unsolved problem whether or not
an odd perfect number exists.  Interest to this problem
has produced many analogous notions and problems concerning
divisors of an integer.
For example, Suryanarayana \cite{Sur} called $N$ to be
superperfect if $\sigma(\sigma(N))=2N$.  It is asked in this paper
and still unsolved whether there were odd superperfect numbers.

Some special classes of divisors have also been studied in several papers.
One of them is the class of unitary divisors defined by Eckford Cohen \cite{CohE}.
A divisor $d$ of $n$ is called a unitary divisor if $(d, n/d)=1$.
Wall \cite{Wal1} introduced the notion of biunitary divisors.
A divisor $d$ of a positive integer $n$ is called a biunitary divisor
if $\gcd(d, n/d)=1$.

Graeme L. Cohen \cite{CohG} generalized these notions.
We call any divisor of a positive integer $n$ to be a $0$-ary divisor of $n$
and a divisor $d$ of a positive integer $n$ to be a $k+1$-ary divisor
if $d$ and $n/d$ does not have a common $k$-ary divisor other than $1$.
Hence $1$-ary divisor is unitary and $2$-ary divisor is biunitary.
We note that a positive integer $d=\prod_i p_i^{f_i}$ with $p_i$ distinct primes and $f_i\geq 0$
is a $k$-ary divisor of $n=\prod_i p_i^{e_i}$ if and only if
$p_i^{f_i}$ is a $k$-ary divisor of $p_i^{e_i}$ for each $i$.
G. L. Cohen \cite[Theorem 1]{CohG} showed that, if $p^f$ is a $e-1$-ary divisor of $p^e$,
then $p^f$ is a $k$-ary divisor of $p^e$ for any $k\geq e-1$
and called such a divisor to be an infinitary divisor.
For any positive integer $n$, a divisor $d=\prod_i p_i^{f_i}$ of $n=\prod_i p_i^{e_i}$
is called to be an infinitary divisor if $p_i^{f_i}$ is an infinitary divisor of $p_i^{e_i}$ for each $i$,
which is written as $d\mid_\infty n$.

According to E. Cohen \cite{CohE}, Wall \cite{Wal1} and G. L. Cohen \cite{CohG} respectively,
henceforth $\sigma^*(N), \sigma^{**}(N)$ and $\sigma_\infty(n)$ denote the sum of unitary, biunitary
and infinitary divisors of $N$.

Replacing $\sigma$ by $\sigma^*$, Subbarao and Warren \cite{SW}
introduced the notion of a unitary perfect number.  $N$
is called to be unitary perfect if $\sigma^*(N)=2N$.  They proved
that there are no odd unitary perfect numbers and $6, 60, 90, 87360$
are the first four unitary perfect numbers.  Later the fifth unitary perfect number
has been found by Wall \cite{Wal2}, but no further instance has been found.
Subbarao \cite{Sub} conjectured that there are only finitely many
unitary perfect numbers.
Similarly, a positive integers $N$ is called biunitary perfect if $\sigma^{**}(N)=2N$.
Wall \cite{Wal1} showed that $6, 60$ and $90$, the first three unitary perfect numbers,
are the only biunitary perfect numbers.

G. Cohen \cite{CohG} introduced the notion of infinitary perfect numbers;
a positive integer $n$ is called infinitary perfect if $\sigma_\infty(n)=2n$.
Cohen \cite[Theorem 16]{CohG} shows that $6, 60$ and $90$, exactly all of the biunitary perfect numbers,
are also all of the infinitary perfect numbers not divisible by $8$.
Cohen gave $14$ infinitary perfect numbers and Pedersen's database,
which is now available at \cite{Moe}, contains $190$ infinitary perfect numbers.

Combining the notion of superperfect numbers and the notion of unitary divisors,
Sitaramaiah and Subbarao \cite{SS} studied unitary superperfect numbers,
integers $N$ satisfying $\sigma^*(\sigma^*(N))=2N$.  They found all unitary superperfect
numbers below $10^8$.  The first ones are $2, 9, 165, 238$.  Thus
there are both even and odd ones.  The author \cite{Ymd1} showed that
$9, 165$ are the all odd ones.

Now we can call an integer $N$ satisfying $\sigma_\infty(\sigma_\infty(N))=2N$
to be infinitary superperfect.
We can see that $2$ and $9$ are infinitary superperfect, while $2$ is also superperfect
(in the ordinary sense) and $9$ is also unitary superperfect.
Below $2^{28}$, we can find some integers $n$ dividing $\sigma_\infty(\sigma_\infty(n))$
but we cannot find any other infinitary superperfect numbers.

Analogous to \cite{Ymd1}, we can show that following result.

\begin{thm}\label{th1}
$9$ is the only odd infinitary superperfect number.
\end{thm}

We can see that this immediately follows from the following result.

\begin{thm}\label{th2}
If $N$ is an infinitary superperfect number with $\omega(\sigma_\infty(N))\leq 2$,
then $N=2$ or $N=9$.
\end{thm}

Indeed, if $N$ is odd and $\sigma_\infty(\sigma_\infty(N))=2N$, then
we have $\omega(N)\leq 2$ as shown in the next section.

Our method does not seem to work to find all odd super perfect numbers
since $\sigma(\sigma(N))=2N$ does not seem to imply that $\omega(\sigma(N))\leq 2$.
Even assuming that $\omega(\sigma(N))\leq 2$, the property of $\sigma$ that
$\sigma(p^e)/p^e>1+1/p$ prevents us from showing that $\sigma(\sigma(N))<2$.
All that we know is the author's result in \cite{Ymd2} that there are only finitely many
odd superperfect numbers $N$ with $\omega(\sigma(N))\leq k$ for each $k$.
For the biunitary analogues, the author \cite{Ymd3} showed that
$2$ and $9$ are the only integers $N$ (even or odd!) such that $\sigma^{**}(\sigma^{**}(N))=2N$.

In Tabel \ref{tbl1}, we give all integers $N\leq 2^{28}$ dividing $\sigma_\infty(\sigma_\infty(N))$.
We found no other infinitary superperfect numbers other than $2$ and $9$,
while we found several integers $N$ dividing $\sigma_\infty(\sigma_\infty(N))$.
From this table, we are led to conjecture that $2$ is the only even infinitary superperfect number.
On the other hand, it seems that for any integer $k\geq 3$, there exist infinitely many integers $N$
for which $\sigma_\infty(\sigma_\infty(N))=kN$.

\section{Preliminary Lemmas}\label{lemmas}
In this section, we shall give several preliminary lemmas concerning the sum of
infinitary divisors used to prove our main theorems.

We begin by introducing Theorem 8 of \cite{CohG}:
writing binary expansions of $e, f$ as $e=\sum_{i\in I} 2^i$
and $f=\sum_{j\in J} 2^j$, $p^f$ is an infinitary divisor of $p^e$
if and only if $J$ is a subset of $I$.
Hence, factoring $n=\prod_{i=1}^r p_i^{e_i}$ and
writing a binary expansion of each $e_i$ as $e_i=\sum_j y_{ij} 2^j$ with $y_{ij}\in \{0, 1\}$,
we observe that, as is shown in \cite{CohG}[Theorem 13],
\begin{equation}\label{eq0}
\sigma_\infty(n)=\prod_{i=1}^r \prod_{y_{ij}=1}\left(1+p_i^{2^j}\right).
\end{equation}

\begin{lem}\label{a}
Factor $n=\prod_{i=1}^r p_i^{e_i}$ and let $l(e_i)$ denote the number of $1$'s in the binary expansion of $e_i$.
$\sigma_\infty(n)$ is divisible by $2$ at least $\sum_{p_i\neq 2}l(e_i)\geq \omega(n)-1$ times.
In particular, $\sigma_\infty(n)$ is odd if and only if $n$ is a power of $2$.
\end{lem}
\begin{proof}
For each prime factor $p_i$, write a binary expansion of each $e_i$ as $e_i=\sum_j y_{ij} 2^j$ with $y_{ij}\in \{0, 1\}$.
Hence $l(e_i)=\sum_j y_{ij}$ holds for each $i$.
Unless $p_i=2$, $p_i^{2^j}+1$ is even for any $j\geq 0$.
By (\ref{eq0}), each product $\sigma_\infty(p_i^{e_i})=\prod_{y_{ij}=1}\left(1+p_i^{2^j}\right)$ except $p_i=2$
is divisible by $2$ at least $l(e_i)$ times and $\sigma_\infty(n)$
at least $\sum_{p_i\neq 2}l(e_i)$ times.
We can easily see that $\sum_{p_i\neq 2}l(e_i)\geq \omega(n)-1$ since $l(m)>0$ for any nonzero integer $m$.
\end{proof}

\begin{lem}\label{b}
If $l>0$, then any prime factor of $2^{2^l}+1$ is congruent to $1\pmod{4}$.
If $l>0$ and $p$ is an odd prime, then
$p^{2^l}+1\equiv 2\pmod{4}$ and any odd prime factor of $p^{2^l}+1$
is congruent to $1\pmod{4}$.
\end{lem}
\begin{proof}
These statements immediately follow from the first supplementary law (see, for example,
Theorem 82 of \cite{HW}) and observing that $x^2+1\equiv 2\pmod{4}$ for any odd integer $x$.
\end{proof}

The following two lemmas follow almost immediately from Bang's result \cite{Ban}.
But we shall include direct proofs.

\begin{lem}\label{c}
If $p$ is a prime and $\sigma_\infty(p^e)$ is a prime power,
then $p$ is a Mersenne prime and $e=1$
or $p=2, e=2^l$ and $\sigma_\infty(p^e)$ is a Fermat prime.
\end{lem}
\begin{proof}
If $p$ is an odd integer and $p^{2^l}+1$ is a prime power,
then $l=0$ and $p+1$ is a power of $2$ since otherwise $p^{2^l}+1\equiv 2\pmod{4}$
cannot be a prime power.

If $2^{2^l}+1=q^f$ for some $q$ and $f>1$, then $q^f-1=2^{2^l}$
and $(q-1)(q^{f-1}+q^{f-2}+\cdots +1)=2^{2^l}$.
Hence both $q-1$ and $q^{f-1}+q^{f-2}+\cdots +1$ must be a power of $2$.
Since $q^{f-1}+q^{f-2}+\cdots +1\equiv f\pmod{q-1}$, $f=2^k$ must be also
a power of $2$. Hence $q^{f-1}+q^{f-2}+\cdots +1=(q+1)(q^2+1)\cdots (q^{2^k-1}+1)$
and each factor must be also a power of $2$.  However, $q^2+1\equiv 2\pmod{4}$
cannot be a power of $2$.  Thus $k=1, f=2$ and both $q-1$ and $q+1$ is a power of $2$.
This implies $q=3$.  However, in this case, $q^f-1=3^2-1=2^3$ is
not of a form $2^{2^l}$.  Hence we cannot have $2^{2^l}+1=q^f$ with $f>1$.
\end{proof}

\begin{lem}\label{d}
$\sigma_\infty(2^e)$ has at least $l(e)$ distinct prime factors.
If $p$ is an odd prime, then $\sigma_\infty(p^e)$ has at least $l(e)+1$ distinct prime factors.
\end{lem}
\begin{proof}
Whether $p$ is odd or two, $\sigma_\infty(p^e)$ is the product
of $l(e)$ distinct numbers of the form $p^{2^l}+1$.
If $k>l$, then $p^{2^k}+1\equiv 2\pmod{p^{2^l}+1}$ and therefore
$p^{2^k}+1$ has a odd prime factor not dividing $p^{2^l}+1$.
\end{proof}

Finally, we shall introduce two technical lemmas needed in the proof.

\begin{lem}\label{e}
If $p^2+1=2q^m$ with $m\geq 2$, then $m$ must be a power of $2$ and,
for any given prime $q$, there exists at most one such $m$.
If $p^{2^k}+1=2q^m$ with $k>1$, then $m=1$.
\end{lem}
\begin{proof}
St{\o}rmer \cite[Th\'{e}or\`{e}me 8]{Stm} shows that
$x^2+1=2y^n$ with $n>1$ odd has no solution in positive integers $x, y$ other than $(x, y)=(1, 1)$.
Hence, in both cases, $m$ must be a power of $2$ including $1$.

For any given prime $q$, applying another result of St{\o}rmer \cite[Th\'{e}or\`{e}me 1]{Stm} to the equation $x^2-2q^2 y^2=-1$
or exploiting a primitive prime factor of a solution of $x^2-2y^2=-1$,
we see that there exists at most one prime $p$ such that $p^2+1=2q^{2^t}$ for some $t>0$.
Now the former statement follows.

Finally, the latter statement follows observing that
$x^4+1=2y^2$, equivalent to $y^4-x^4=(y^2-1)^2$, has no solution other than
$(1, 1)$ by Fermat's well-known right triangle theorem.
\end{proof}

\begin{lem}\label{f}
If $p, q$ are odd primes satisfying $p^{2^k}+1=2q$ and $2^{2^{k+1}}\equiv 1\pmod{q}$ with $k>0$,
then $(p, q)=(3, 5)$ and $k=1$.
\end{lem}
\begin{proof}
Since $p^{2^{k+1}}\equiv 2^{2^{k+1}}\equiv 1\pmod{q}$,
any integer $X$ of the form $2^a p^b$ satisfies the congruence $X^{2^{k+1}}\equiv 1\pmod{q}$.

If $p>3$, then $2p^{2^k-1}<p^{2^k}/2<q$
and therefore the congruence $X^{2^{k+1}}\equiv 1\pmod{q}$
must have at least $2^{k+1}+1$ solutions
\begin{equation*}
X=1, p, \ldots, p^{2^k-1}, 2, 2p, \ldots, 2p^{2^k-1}, 2^2
\end{equation*}
in the range $1\leq X\leq q-1$, which is a contradiction.

If $p=3$, then the congruence
$X^{2^{k+1}}\equiv 1\pmod{q}$ has $2^{k+1}$ solutions
$1, p, \ldots, p^{2^k-1}, 2, 2p, \ldots, 2p^{2^k-2}, 2^2$
with $1\leq X\leq q-1$.
Since $X^{2^{k+1}}\equiv 1\pmod{q}$ can have no more solution in the range $1\leq X\leq q-1$,
$2p^{2^k-1}$ must be congruent to one of these $2^{k+1}$ numbers modulo $q$.
Hence we see that $2p^{2^k-1}-2^a p^b\equiv 0\pmod{q}$ for some $2^a p^b$ in the above $2^{k+1}$ solutions.
If $b>0$, then $2p^{2^k-2}-2^a p^{b-1}\equiv 0\pmod{q}$.
Since $2p^{2^k-2}<4q/p^2<q$, we have $2p^{2^k-2}=2^a p^{b-1}$ and
$a=1, b=2^k-1$, which is a contradiction since $(a, b)=(1, 2^k-1)$ is not listed above.

If $a>0$, we have $p^{2^k-1}-2^{a-1} p^b\equiv 0\pmod{q}$.
Since $p^{2^k-1}<2q/p<q$, we have $p^{2^k-1}=2^{a-1}q^b$ and therefore
$a=1, b=2^k-1$, which is a contradiction same as above.

Hence we have $a=b=0$ and $2p^{2^k-1}\equiv 1\pmod{q}$.
Write $2p^{2^k-1}=hq+1$.
We have $4q=2(p^{2^k}+1)=2p^{2^k-1}p+2=p(hq+1)+2=hpq+p+2$
and therefore $h=1, q=5$ and $k=1$, which gives $(p, q)=(3, 5)$.
\end{proof}

\section[Proofs of Theorems 1.1 and 1.2]{Proofs of Theorems \ref{th1} and \ref{th2}}

We begin by noting that Theorem \ref{th1} follows from Theorem \ref{th2}.
Indeed, if $N$ is odd and $\sigma_\infty(\sigma_\infty(N))=2N$,
then Lemma \ref{a} gives that $\omega(\sigma_\infty(N))\leq 2$ and therefore Theorem \ref{th2}
would yield Theorem \ref{th1}.

In order to prove Theorem \ref{th2}, we shall first show that
if $\sigma_\infty(N)$ is odd or a prime power, then $N$ must be $2$.
If $\sigma_\infty(N)$ is a prime power, then
Lemma \ref{c} immediately yields that $N=2^e$ or $\sigma_\infty(N)$ is a power of $2$,
where the latter case cannot occur since $\sigma_\infty(\sigma_\infty(N))$
must be odd in the latter case while we must have $\sigma_\infty(\sigma_\infty(N))=2N$.
If $\sigma_\infty(N)$ is odd, then $N$ must be a power of $2$ by Lemma \ref{a}.

Thus we see that if $\sigma_\infty(N)$ is odd or a prime power, then $N=2^e$
must be a power of $2$.
Now we can easily see that $\sigma_\infty(\sigma_\infty(N))=2N=2^{e+1}$ must also be a power of $2$.
Hence, for each prime-power factor $q_i^{f_i}$ of $\sigma_\infty(N)$,
$\sigma_\infty(q_i^{f_i})$ is also a power of $2$.
By Lemma \ref{c}, each $f_i=1$ and $q_i$ is a Mersenne prime.
Hence we see that $\sigma_\infty(N)=\sigma_\infty(2^e)$ must be a product of Mersenne primes.
Let $r$ be an integer such that $2^{2^r}\mid_\infty N$.
Then $2^{2^r}+1$ must also be a product of Mersenne primes.
By Lemma \ref{b}, only $r=0$ is appropriate and therefore $e=0$.
Thus we conclude that if $\sigma_\infty(N)$ is odd or a prime power, then $N=2$.

Henceforth we are interested in the case $\sigma_\infty(N)=2^f q^{2^l}$ with $f>0$ and $l\geq 0$.
Factor $N=\prod_i p_i^{e_i}$.
Our proof proceeds as follows:
(I) if $l=0$, then there exists exactly one prime factor $p_i$ of $N$
such that $q$ divides $\sigma(p_i^{e_i})$, (IA) if $l=0$ and $f=1$, then $N=9$, (IBa) it is impossible
that $l=0, f>1$ and $p_i\mid q+1$, (IBb) it is impossible that $l=0, f>1$ and
$p_i$ does not divide $q+1$, (II) if $l>0$, then there exists at most one prime factor $p_i$ of $q^{2^l}+1$
such that $p_i^{2^k}+1=2q$, (IIa) it is impossible that $q^{2^l}+1$ has no such prime factor,
(IIb) it is impossible that $q^{2^l}+1$ has one such prime factor $p_i$.

First we shall settle the case $l=0$, that is, $\sigma_\infty(N)=2^f q$.
Since $q$ divides $N$ exactly once, there exists exactly one index $i$ such that
$q$ divides $\sigma_\infty(p_i^{e_i})$.

For any index $j$ other than $i$, we must have $\sigma_\infty(p_j^{e_j})=2^{k_j}$
and therefore, by Lemma \ref{c}, we have $e_j=1$ and $p_j=2^{k_j}-1$ for some intger $k_j$.
Clearly $p_j$ must divide $2N=\sigma_\infty(\sigma_\infty(N))=\sigma_\infty(2^f)(q+1)$
and Lemma \ref{b} yields that $p_j\mid (q+1)$ unless $p_j=3$.

If $f=1$, then $\omega(N)=1$ by Lemma \ref{a} and $2N=\sigma_\infty(2q)=3(q+1)$.
Hence $N=3^e$ and $\sigma_\infty(3^e)=2q$.
By Lemma \ref{a}, we have $e=2^u$ and $3^e+1=2q=2(2\times 3^{e-1}-1)=4\times 3^{e-1}-2$.
Hence $3^{e-1}=3$, that is, $N=9$ and $q=5$.  This gives an infinitary superperfect number $9$.

If $f>1$, then, by Lemma \ref{b}, $\sigma(2^f)$ must have at least one prime factor congruent to $1\pmod{4}$,
which must be $p_i$.  Hence $f=2^m$ and $\sigma_\infty(2^f)=p_i$
or $f=2^m+1$ and $\sigma_\infty(2^f)=3p_i$.
Indeed, if $2^{2^v}\mid_\infty 2^f$ and $v>0$, then $2^{2^v}+1\mid\sigma_\infty(2^f)\mid \sigma_\infty(\sigma_\infty(N))=2N$.
By Lemma \ref{b}, a prime factor of $2^{2^v}+1$ must be congruent to $1\pmod{4}$
and therefore $2^{2^v}+1$ must be a power of $p_i$.
But, by Lemma \ref{c}, we must have $2^{2^v}+1=p_i$.

If $p_i$ divides $q+1$, then $e_i\geq 2$.
By Lemma \ref{d}, we must have $e_i=2^v$ and $p_i^{e_i}+1=2q$.
Since $p_i=\sigma_\infty(2^f)$, $p_i^{e_i-1}$ divides $q+1$ and therefore $2(q+1)=p_i^{e_i}+3$,
Hence $p_i^{e_i}\equiv -3\pmod{p_i^{e_i-1}}$, which is impossible since $p_i>3$ now.

If $p_i$ does not divide $q+1$, then $e_i=1$
and $2^{k_i} q=p_i+1=2^{2^m}+2$.
Hence $k_i=1$ and $q=2^{2^m-1}+1$.  Now $m=1$ with $q=3$
is the only $m$ such that $q$ is prime.
Hence $f=2$ or $f=3$.
If $f=2$, then $\sigma_\infty(N)=12, 2N=\sigma_\infty(\sigma_\infty(N))=20, N=10$ and $\sigma_\infty(N)=18$,
which is a contradiction.
Similarly, if $f=3$, then $\sigma_\infty(N)=24, 2N=\sigma_\infty(\sigma_\infty(N))=60$ and $N=30>24=\sigma_\infty(N)$,
leading to a contradiction again.  Thus the case $\sigma_\infty(N)=2^f q$ with $f>1$
has turned out to be impossible and $N=3^2$ is the only infinitary superperfect number
with $\sigma_\infty(N)=2q$.

Now the remaining case is $N=2^f q^{2^l}$ with $l>0$.
If $p$ is odd and divides $\sigma_\infty(q^{2^l})=q^{2^l}+1$, then
$p$ divides $\sigma_\infty(\sigma_\infty(N))=2N$ and therefore $p$ divides $N$.
If $p^{2^k}\mid_\infty N$, then $p^{2^k}+1$ divides $\sigma_\infty(N)=2^f q^{2^l}$
and therefore we can write $p^{2^k}+1=2 q^t$.
We note that $p\equiv 1\pmod{4}$ since $p$ is odd and divides $q^{2^l}+1$ with $l>0$.
Hence we see that a) if $k=0$, then $p+1=2q^t$, b) if $k=1$, then $p^2+1=2q$ or $2q^{2^u}$ by Lemma \ref{e}
and c) if $k>1$, then $p^{2^k}+1=2q$ by Lemma \ref{e}.

Clearly, there exists at most one prime $p_i$ such that $p^{2^k}+1=2q$ for some integer $k>0$.
Moreover, by Lemma \ref{e}, there exists at most one prime $p_j$ such that $p_j^2+1=2q^{2^u}$ for some integer $u>0$.
Letting $i$ and $j$ denote the indices of such primes respectively if these exist,
$q^{2^l}+1$ can be written in the form
\begin{equation}
q^{2^l}+1=2p_i^{g_i} p_j^{g_j}(2q^{t_1}-1)(2q^{t_2}-1)...,
\end{equation}
where $g_i, g_j\geq 0$ may be zero.

If $g_i\neq 0$, then we have $2p_i^{g_i}p_j^{g_j}\equiv \pm 1\pmod{q}$ and therefore,
observing that $p_i^{2^{k+1}}\equiv p_j^4\equiv 1\pmod{q}$, we have $2^{2^{k+1}}\equiv 1\pmod{q}$.
By Lemma \ref{f}, we must have $p_i=3, e_i=2$ and $q=5$ and $p_j$ cannot exist.
Since $p_i=3$ divides $q^{2^l}+1$, we must have $l=0$, contrary to the assumption $l>0$.

If $g_i=0$, then we must have $2p_j^{g_j}\equiv \pm 1\pmod{q}$.
If $g_j$ is even, then $2p_j^{g_j}\equiv 2(-1)^{g_j/2}\equiv \pm 2\pmod{q}$ cannot be $\pm 1\pmod{q}$.
Hence $g_j$ must be odd and $2p_j\equiv \pm 1\pmod{q}$.
Since $p_j^4\equiv 1\pmod{q}$, we have $2^4\equiv 1\pmod{q}$ and $q\equiv 1\pmod{4}$.
Equivalently, we have $q=5$ and therefore $p_j^2+1=2\times 5^{2^k}$ with $k>0$.
From a result of St{\o}rmer \cite[p. 26]{Stm}, it follows that $p_j=7$ and $k=1$.  However, this is impossible
since $p_j$ divides neither $\sigma_\infty(5^2)=2\times 13$
nor $\sigma_\infty(2^f)$ by Lemma \ref{b}.  Now our proof is complete.

\begin{table}
\caption{All integers $N\leq 2^{28}$ for which $\sigma_\infty(\sigma_\infty(N))=kN$}\label{tbl1}
\begin{center}
\begin{small}
\begin{tabular}{| l | c |}
 \hline
$N$ & $k$ \\
 \hline
$1$ & $1$ \\
$2$ & $2$ \\
$8=2^3$ & $3$ \\
$9=3^2$ & $2$ \\
$10=2\cdot 5$ & $3$ \\
$15=3\cdot 5$ & $4$ \\
$18=2\cdot 3^2$ & $4$ \\
$24=2^3\cdot 3$ & $5$ \\
$30=2\cdot 3\cdot 5$ & $5$ \\
$60=2^2\cdot3\cdot 5$ & $6$ \\
$720=2^4\cdot 3^2\cdot 5$ & $3$ \\
$1020=2^2\cdot 3\cdot 5\cdot 17$ & $4$ \\
$4080=2^4\cdot 3\cdot 5\cdot 17$ & $3$ \\
$8925=3\cdot 5^2\cdot 7\cdot 17$ & $4$ \\
$14688=2^5\cdot 3^3\cdot 17$ & $5$ \\
$14976=2^7\cdot 3^2\cdot 13$ & $5$ \\
$16728=2^3\cdot 3\cdot 17\cdot 41$ & $4$ \\
$17850=2\cdot 3\cdot 5^2\cdot 7\cdot 17$ & $8$ \\
$35700=2^2\cdot 3\cdot 5^2\cdot 7\cdot 17$ & $6$ \\
$36720=2^4\cdot 3^3\cdot 5\cdot 17$ & $6$ \\
$37440=2^6\cdot 3^2\cdot 5\cdot 13$ & $6$ \\
$66912=2^5\cdot 3\cdot 17\cdot 41$ & $3$ \\
$71400=2^3\cdot 3\cdot 5^2\cdot 7\cdot 17$ & $12$ \\
$285600=2^5\cdot 3\cdot 5^2\cdot 7\cdot 17$ & $9$ \\
$308448=2^5\cdot 3^4\cdot 7\cdot 17$ & $5$ \\
$381888=2^6\cdot 3^3\cdot 13\cdot 17$ & $5$ \\
 \hline
\end{tabular}
\end{small}
\begin{small}
\begin{tabular}{| l | c |}
 \hline
$N$ & $k$ \\
 \hline
$428400=2^4\cdot 3^2\cdot 5^2\cdot 7\cdot 17$ & $3$ \\
$602208=2^5\cdot 3^3\cdot 17\cdot 41$ & $6$ \\
$636480=2^6\cdot 3^2\cdot 5\cdot 13\cdot 17$ & $4$ \\
$763776=2^7\cdot 3^3\cdot 13\cdot 17$ & $10$ \\
$856800=2^5\cdot 3^2\cdot 5^2\cdot 7$ & $6$ \\
$1321920=2^6\cdot 5^5\cdot 5\cdot 17$ & $7$ \\
$1505520=2^4\cdot 3^3\cdot 5\cdot 17\cdot 41$ & $4$ \\
$3011040=2^5\cdot 3^3\cdot 5\cdot 17\cdot 41$ & $8$ \\
$3084480=2^6\cdot 3^4\cdot 5\cdot 7\cdot 17$ & $5$ \\
$21679488=2^7\cdot 3^5\cdot 17\cdot 41$ & $7$ \\
$22276800=2^6\cdot 3^2\cdot 5^2\cdot 7\cdot 13\cdot 17$ & $6$ \\
$30844800=2^{10}\cdot 3^4\cdot 5^3\cdot 7\cdot 17$ & $7$ \\
$31615920=2^4\cdot 3^4\cdot 5\cdot 7\cdot 17\cdot 41$ & $4$ \\
$44553600=2^7\cdot 3^2\cdot 5^2\cdot 7\cdot 13\cdot 17$ & $12$ \\
$50585472=2^7\cdot 3^4\cdot 7\cdot 17\cdot 41$ & $5$ \\
$63231840=2^5\cdot 3^4\cdot 5\cdot 7\cdot 17\cdot 41$ & $8$ \\
$126463680=2^6\cdot 3^4\cdot 5\cdot 7\cdot 17\cdot 41$ & $6$ \\
$213721200=2^4\cdot 3^3\cdot 5^2\cdot 7\cdot 11\cdot 257$ & $4$ \\
$230177280=2^9\cdot 3\cdot 5\cdot 17\cdot 41\cdot 43$ & $9$ \\
$252927360=2^7\cdot 3^4\cdot 5\cdot 7\cdot 17\cdot 41$ & $12$ \\
$307758528=2^6\cdot 3^5\cdot 7\cdot 11\cdot 257$ & $5$ \\
$345265920=2^8\cdot 3^2\cdot 5\cdot 17\cdot 41\cdot 43$ & $3$ \\
$427442400=2^5\cdot 3^3\cdot 5^2\cdot 7\cdot 11\cdot 257$ & $8$ \\
$437898240=2^{10}\cdot 3^2\cdot 5\cdot 13\cdot 17\cdot 43$ & $5$ \\
$466794240=2^8\cdot 3\cdot 5\cdot 11\cdot 43\cdot 257$ & $3$ \\
$512930880=2^6\cdot 3^4\cdot 5\cdot 7\cdot 11\cdot 257$ & $4$ \\
 \hline
\end{tabular}
\end{small}
\end{center}
\end{table}

{}
\vskip 12pt

{\small Center for Japanese language and culture, Osaka University,\\ 562-8558, 8-1-1, Aomatanihigashi, Minoo, Osaka, Japan}\\
{\small e-mail: \protect\normalfont\ttfamily{tyamada1093@gmail.com} URL: \url{http://tyamada1093.web.fc2.com/math/}

\begin{thebibliography}{}
\bibitem{Ban}
A. S. Bang, Taltheoretiske Unders{\o}gelser, {\it Tidsskrift Math.} \textbf{5 {IV}} (1886), 70--80 and 130--137.
\bibitem{CohE}
Eckford Cohen, Arithmetical functions associated with the unitary divisors of an integer, {\it Math. Z.} \textbf{74} (1960), 66--80.
\bibitem{CohG}
Graeme L. Cohen, On an integer's infinitary divisors, {\it Math. Comp.} \textbf{54} (1990), 395--411.
\bibitem{HW}
G. H. Hardy and E. M. Wright, revised by D. R. Heath-Brown and J. H. Silverman,
{\it An Introduction to the Theory of Numbers}, Sixth edition, Oxford University Press, Oxford, 2008.
\bibitem{Moe}
David Moews, A database of aliquot cycles, \url{http://djm.cc/aliquot-database/aliquot-database.uhtml}.
\bibitem{SS}
V. Sitaramaiah and M. V. Subbarao, On the equation $\sigma^*(\sigma^*(N))=2N$, {\it Util. Math.} \textbf{53}(1998), 101--124.
\bibitem{Stm}
Carl St{\o}rmer, Quelques th\'{e}or\`{e}mes sur l'\'{e}quation de Pell $x^2-Dy^2=\pm 1$ et leurs applications,
Skrift. Vidensk. Christiania I. Math. -naturv. Klasse (1897), Nr. 2, 48 pages.
\bibitem{Sub}
M. V. Subbarao, Are there an infinity of unitary perfect numbers?, {\it Amer. Math. Monthly} \textbf{77} (1970), 389--390.
\bibitem{SW}
M. V. Subbarao and L. J. Warren, Unitary perfect numbers, {\it Canad. Math. Bull.} \textbf{9} (1966), 147--153.
\bibitem{Sur}
D. Suryanarayana, Super perfect numbers, {\it Elem. Math.} \textbf{24} (1969), 16--17.
\bibitem{Wal1}
Charles R. Wall, Bi-unitary perfect numbers, {\it Proc. Amer. Math. Soc.} \textbf{33} (1972), 39--42.
\bibitem{Wal2}
Charles R. Wall, The fifth unitary perfect number, {\it Canad. Math. Bull.} \textbf{18} (1975), 115--122.
\bibitem{Ymd1}
T. Yamada, Unitary super perfect numbers, {\it Math. Pannon.} \textbf{19} (2008), 37--47.
\bibitem{Ymd2}
T. Yamada, On finiteness of odd superperfect numbers, \url{https://arxiv.org/abs/0803.0437}.
\bibitem{Ymd3}
T. Yamada, $2$ and $9$ are the only biunitary superperfect numbetrs, \url{https://arxiv.org/abs/1705.00189}.
\end{thebibliography}
\end{document}